\documentclass[graybox]{svmult}

\usepackage{courier}        
\usepackage{makeidx}         
\usepackage{graphicx}        
\usepackage{multicol}        
\usepackage[bottom]{footmisc}
\usepackage{amssymb, amsmath, amscd}

\newcommand{\C}{{\mathbb{C}}}
\newcommand{\R}{{\mathbb{R}}}
\newcommand{\Z}{{\mathbb{Z}}}

\makeindex             

\begin{document}

\title*{The Teichm\"uller Stack}
\author{Laurent MEERSSEMAN}
\institute{L. Meersseman \at LAREMA, UMR 6093 CNRS, Facult\'e des Sciences, 2 Boulevard Lavoisier 49045 Angers cedex 01, \email{laurent.meersseman@univ-angers.fr}}
%
%
\maketitle


\abstract{This paper is a comprehensive introduction to the results of \cite{LMStacks}. It grew as an expanded version of a talk given at INdAM Meeting Complex and Symplectic Geometry, held at Cortona in June 12-18, 2016. It deals with the construction of the Teichm\"uller space of a smooth compact manifold $M$ (that is the space of isomorphism classes of complex structures on $M$) in arbitrary dimension. The main problem is that, whenever we leave the world of surfaces, the Teichm\"uller space is no more a complex manifold or an analytic space but an analytic Artin stack. 
We explain how to construct explicitly an atlas for this stack using ideas coming from foliation theory. 
Throughout the article, we use the case of $\mathbb S^3\times\mathbb S^1$ as a recurrent example.  }

\section{Introduction}
\label{Intro}

Let $M$ be a compact $C^\infty$ connected oriented manifold. Assume that $M$ is even-dimensional and admits complex structures. We are interested in the Teichm\"uller space $\mathcal T(M)$. To define it, we start with the {\it moduli space} $\mathcal M(M)$ of complex structures on $M$. We can formally define it as the set of complex manifolds diffeomorphic to $M$ up to biholomorphisms. In short, 

\begin{equation}
\label{M}
\mathcal M(M)=\{X\text{ complex manifold }\mid X\simeq_{so} M\}/\sim
\end{equation}
where $X\simeq_{so}M$ means that there exists a $C^\infty$-diffeomorphism from $X$ to $M$ preserving the orientations and where $X\sim Y$ if they are biholomorphic.

Thanks to Newlander-Nirenberg Theorem \cite{NN}, a structure of a complex manifold on $M$ is equivalent to an integrable complex operator $J$ on $M$, that is a $C^\infty$ bundle operator $J$ on the tangent bundle $TM$ such that 
\begin{equation}
\label{ico}
J^2=- Id\qquad\text{ and }\qquad [T^{0,1},T^{0,1}]\subset T^{0,1}
\end{equation}
for
\begin{equation}
\label{T01}
T^{0,1}=\{v+iJv\mid v\in TM\otimes\mathbb C\}
\end{equation}
the subbundle of the complexified tangent bundle $TM\otimes\mathbb C$ formed by the eigenvectors of $J$ with eigenvalue $-i$. Here of course, $J$ has been linearly extended to the complexified tangent bundle. We may thus rewrite \eqref{M} as
\begin{equation}
\label{M2}
\mathcal M(M)=\{J\text{ o.p. integrable complex operator on } M\}/\sim
\end{equation}
where o.p. means orientation preserving, i.e. the orientation induced on $M$ by $J$ coincides with that of $M$.

Now, it is easy to check that $(M,J)$ and $(M,J')$ are biholomorphic if and only if there exists a diffeomorphism $f$ of $M$ whose differential $df$ satisfies
\begin{equation}
\label{Jf}
J'=(df)^{-1}\circ J\circ df
\end{equation}
In other words, denoting $J\cdot f$ the right hand side of \eqref{Jf}, we see that \eqref{Jf} defines an action of the diffeomorphism group $\text{Diff}(M)$ onto $\mathcal I(M)$, the set of o.p. integrable complex operators appearing in \eqref{M2}. Since our operators are o.p., this is even in fact an action of $\text{Diff}^+(M)$, the group of diffeomorphisms of $M$ that preserve the orientation.

So we end with
\begin{equation}
\label{M3}
\mathcal M(M)=\mathcal I(M)/\text{Diff}^+(M)
\end{equation}
and we are in position to define the {\it Teichm\"uller space} of $M$ as 
\begin{equation}
\label{T}
\mathcal T(M)=\mathcal I(M)/\text{Diff}^0(M)
\end{equation}
where $\text{Diff}^0(M)$ is the group of diffeomorphisms $C^\infty$-isotopic to the identity, that is the connected component of the identity in $\text{Diff}^+(M)$.

It is important to notice that \eqref{M3} and \eqref{T} define {\it topological} spaces and not only sets as \eqref{M}. Indeed, we endow $\mathcal I(M)$ and $\text{Diff}(M)$ with the topology of uniform convergence of sequences of operators/functions and all their derivatives (the $C^\infty$-topology) and we endow   \eqref{M3} and \eqref{T} with the quotient topology.

In fact, more can be said. Replacing $C^\infty$ operators, resp.  $C^\infty$ functions with Sobolev $L^2_l$ operators, resp. $L^2_{l+1}$ functions (for $l$ big), then $\mathcal I(M)$ is a Banach complex analytic space in the sense of \cite{Do}. Also $\text{Diff}^0(M)$ is a complex Hilbert manifold and acts by holomorphic transformations on $\mathcal I(M)$. 

\begin{remark}
\label{cc}
The Teichm\"uller space may have several/countably many connected components (defined as the quotient of a connected component of the space of operators $\mathcal I(M)$ by the $\text{Diff}^0(M)$ action). We will always consider a single connected component of it.
\end{remark}

For $M$ a smooth surface, this definition of Teichm\"uller space coincides with the usual one. Then $\mathcal T(M)$ is a complex manifold and enjoys wonderful properties such as the existence of several nice metrics or of explicit interesting compactifications. The situation is completely different in the higher dimensional case.
It is known since quite a long time, at least since the first works of Kodaira-Spencer at the end of the fifties, that the Teichm\"uller space is not even a complex analytic space in general. In order to put an analytic structure in some sense on the space $\mathcal T(M)$, one has to consider it as an analytic stack. And in order to make this stack structure concrete and useful, one has to give an explicit atlas of it.
  
This was done in \cite{LMStacks}. The crucial idea is to understand the action of  $\text{Diff}^0(M)$ onto $\mathcal I$ as defining a foliation (in a generalized sense) on $\mathcal I$. Hence $\mathcal T(M)$ is the leaf space of this foliation, so as a stack, an atlas can be obtained as a (generalized) holonomy groupoid for this foliation. The aim of this paper is to serve as a comprehensive survey of \cite{LMStacks}, putting emphasis on the main ideas, on examples and on applications. Only section  \ref{appli} contains new results: we briefly report on work in progress by C. Fromenteau.

\section{Examples}
\label{ex}
\begin{example}
\label{ec}

To begin with, let us consider the Teichm\"uller space of $\mathbb S^1\times\mathbb S^1$. By Riemann's uniformization theorem, every Riemann surface diffeomorphic to $\mathbb S^1\times\mathbb S^1$ is a compact complex torus, that is the quotient of $\mathbb C$ by a lattice $\mathbb Zv\oplus\mathbb Zw$ with $(v,w)$ a direct $\mathbb R$-basis\footnote{It should be noted that, in the definition of Teichm\"uller space, we consider {\it any} complex structure, not only projective ones. Hence, in the case of $\mathbb S^1\times\mathbb S^1$, we {\it have to prove} that every such structure is projective.}. Indeed, its universal covering cannot be $\mathbb P^1$ for topological reasons, and the case of $\mathbb D$ is discarded because its automorphism group does not contain any $\mathbb Z^2$ subgroup acting freely.

A classical computation shows that the lattice can be assumed to be of the form $\mathbb Z\oplus \mathbb Z\tau$ with $\tau$ belonging to the upper half-plane $\mathbb H$. Moreover, two such lattices give rise to biholomorphic tori if and only they are related through the formula
\begin{equation}
\label{SL2action}
\tau'=\dfrac{a\tau+b}{c\tau+d}=:A\cdot \tau\qquad\text{ with }\qquad A=\begin{pmatrix}
a &b\\
c &d
\end{pmatrix}
\in\text{SL}_2(\mathbb Z)
\end{equation}
Hence, the moduli space $\mathcal M(\mathbb S^1\times\mathbb S^1)$ is the complex orbifold $\mathbb H/\text{SL}_2(\mathbb Z)$, the action being defined through \eqref{SL2action}.

However, the Teichm\"uller space $\mathcal T(\mathbb S^1\times\mathbb S^1)$ is just $\mathbb H$. This is because a non-trivial element $A$ of $\text{SL}_2(\mathbb Z)$ sends the lattice associated to $\tau$ isomorphically onto the lattice associated to $\tau'$ (we use the same notations as in \eqref{SL2action}), but does not send $1$ to $1$ and $\tau$ to $\tau'$. Since $1$ and $\tau$ in $\mathbb C$ descends as two loops on $\mathbb S^1\times\mathbb S^1$ which generates a basis of the first homology group with values in $\mathbb Z$, and since $1$ and $\tau'$ defines the same basis, this means that the biholomorphism induced by $A$ does not act trivially in homology. So it cannot be isotopic to the identity and $\tau$ differs from $A\cdot \tau$ in the Teichm\"uller space.
\end{example}

The Teichm\"uller space of Example \ref{ec} has the wonderful property of being a complex manifold. Moreover, this complex structure is natural in the sense that it is "compatible" with deformation theory. More precisely, every deformation of complex tori, that is every smooth morphism $\mathcal X\to B$ with fibers diffeomorphic to $\mathbb S^1\times\mathbb S^1$ defines a mapping from the parameter space $B$ to $\mathcal T(\mathbb S^1\times\mathbb S^1)$. The point here is that this mapping is {\it holomorphic}.

The fundamental question is to know if this is a general property of Teichm\"uller spaces or something specific to dimension one. Surely a Teichm\"uller space has to be a complex object so we ask

\medskip
\noindent {\it Is it possible to endow any Teichm\"uller space with the structure of an analytic space?}
\medskip

There are many cases for which this is true at least locally. For example, in \cite{CataneseSurvey}, Catanese shows this is ok locally for K\"ahler manifolds with trivial or torsion canonical bundle and for minimal surfaces of general type with no automorphisms or rigidified (i.e. with no automorphism smoothly isotopic to the identity)  with ample canonical bundle.

Recall that an analytic space is a Hausdorff topological space locally modelled onto the zero set of holomorphic functions and that the Hausdorffness requirement does not follow from the local models. There exist non-Hausdorff analytic spaces, as well as non-Hausdorff manifolds, that is objects having the right local model but not separated as topological spaces. More important for us, this is exactly what happens for the Teichm\"uller space of irreducible Hyperk\"ahler manifolds\footnote{Here this is defined as the space of isomorphism classes of Hyperk\"ahler complex structures on a fixed smooth manifold, and not of arbitrary complex structures.} \cite{VerbitskyHK}. So we should allow non-Hausdorff analytic spaces to expect a positive answer to our question.

Nevertheless, this is indeed not enough. The situation is even worse than that, as shown by the following example.

\begin{example}
\label{Hopfexample}
Our guiding example is that of $\mathbb S^3\times \mathbb S^1$. It was proved by Kodaira in \cite{Kod} that any complex surface homeomorphic to $\mathbb S^3\times \mathbb S^1$ is a {\it Hopf surface}, that is the quotient of $\mathbb C^2\setminus\{(0,0)\}$ by the group generated by a holomorphic contraction of $\mathbb C^2$. Besides, every such contraction can be either linearized and then diagonalized with eigenvalues of modulus strictly less than $1$, or reduced to the following resonant normal form
\begin{equation}
\label{Hopfnf}
(z,w)\longmapsto  g_{\lambda,p}(z,w):=(\lambda z+w^p,\lambda^p w)\qquad\text{ for }p\in\mathbb N^*,\ 0<\vert\lambda\vert<1
\end{equation}
Note that $p=1$ corresponds to the linear but non diagonalizable case, whereas $p>1$ corresponds to non-linear cases. As a consequence, the classification of Hopf surfaces up to biholomorphism is as follows. Let $X_g$ be such a surface defined by the contracting biholomorphism $g$ of $\mathbb C^2$. Let $\lambda_1$ and $\lambda_2$ be the eigenvalues of the linear part of $g$ at $(0,0)$ and assume that $\vert\lambda_2\vert\leq \vert\lambda_1\vert$. Then,
\begin{enumerate}
\item If there is no resonance, that is if $\lambda_2^p$ is different from $\lambda_1$ for all $p\in\mathbb N^*$, then $X_g$ is biholomorphic to $X_{\left (\begin{smallmatrix} \lambda_1 &0\\ 0 &\lambda_2\end{smallmatrix}\right )}$.
\item If there is a resonance of order $p$, then $X_g$ is biholomorphic either to $X_{\left (\begin{smallmatrix} \lambda_1 &0\\ 0 &\lambda_2\end{smallmatrix}\right )}$ or to $X_{g_{\lambda,p}}$.
\end{enumerate}
All these models are pairwise non biholomorphic. As a consequence, one can show that $\mathcal T(\mathbb S^3\times \mathbb S^1)$, as a topological space, is as follows\footnote{Recall that we just consider one connected component of the Teichm\"uller space. Here, the Teichm\"uller space has several connected components, cf. the discussion in \cite{LMStacks}.} \cite{LMStacks}. Let
\begin{equation}
\label{detrace}
\begin{CD}
A\in\text{GL}_2^c(\C)@>\pi >> (\det A,\text{Tr }A)\in\C^2
\end{CD}
\end{equation}
where the superscript $c$ means that we only consider contracting matrices. The image of the map $\pi$ in \eqref{detrace}, say $\mathcal D$, is a bounded domain in $\C^2$ that is exactly the Teichm\"uller space of linear diagonal Hopf surfaces.

To add the linear but non diagonalizable case, one has to add a non-separated copy $\mathcal C$ of the curve 
\begin{equation}
\label{nscurve}
\pi_*\{A\in \text{GL}_2^c(\C)\mid 4\det A=(\text{Tr }A)^2\}
\end{equation} 
in $\mathcal D$. This is because a point in this curve corresponds to two non biholomorphic Hopf surfaces. Hence we distinguish $(\lambda^2,2\lambda)\in\mathcal D$ which encodes $X_{\left (\begin{smallmatrix} \lambda &0\\ 0 &\lambda\end{smallmatrix}\right )}$ and the same point in the added curve $\mathcal C$ which encodes 
$X_{\left (\begin{smallmatrix} \lambda &1\\ 0 &\lambda\end{smallmatrix}\right )}$. 
The augmented domain, say $\mathcal D_{\mathcal C}$, is the Teichm\"uller space of linear Hopf surfaces and has the topology of the conjugacy classes of contracting matrices. In particular, a sufficiently small neighborhood of a point $\left (\begin{smallmatrix} \lambda &1\\ 0 &\lambda\end{smallmatrix}\right )$ of $\mathcal C$ in $\mathcal D_{\mathcal C}$ does not contain the corresponding point $\left (\begin{smallmatrix} \lambda &1\\ 0 &\lambda\end{smallmatrix}\right )$ of $\mathcal D$, whereas every neighborhood of $\left (\begin{smallmatrix} \lambda &1\\ 0 &\lambda\end{smallmatrix}\right )\in\mathcal D$ contains $\left (\begin{smallmatrix} \lambda &1\\ 0 &\lambda\end{smallmatrix}\right )\in\mathcal C$.

But we are not done, since the resonant non linear Hopf surfaces are missing. To include them, for each $p>1$, we add a non-separated copy $\mathcal C_p$ of the curve
\begin{equation}
\label{nscurvep}
\{(\lambda^{p+1},\lambda+\lambda^p)\mid 0<\vert\lambda\vert<1\}
\end{equation} 
to encode the contractions $g_{p,\lambda}$. Thus, a point $(\lambda^{p+1},\lambda+\lambda^p)$ in $\mathcal D$ encodes $X_{\left (\begin{smallmatrix} \lambda &0\\ 0 &\lambda^p\end{smallmatrix}\right )}$ whereas its double in $\mathcal C_p$ encodes $X_{g_{\lambda,p}}$.

We therefore obtain finally a non-Hausdorff space
\begin{equation}
\label{topoTeichHopf}
\mathcal D_{(\mathcal C)}=\mathcal D_{\mathcal C}\sqcup \mathcal C_2\sqcup\hdots
\end{equation}
which looks like a bounded domain in $\C^2$ plus a countable number of pairwise disjoint analytic curves and is endowed with the previously explained topology. The space is not Hausdorff along these curves, since a point in such a curve and the corresponding point in $\mathcal D_{\mathcal C}$ cannot be separated. However, two such points do not play symmetric roles. Every open neighborhood of $(\lambda^{p+1},\lambda+\lambda^p)$ in $\mathcal D_{\mathcal C}$ contains the corresponding point of $\mathcal C_p$, whereas a sufficiently small neighborhood of $(\lambda^{p+1},\lambda+\lambda^p)$ in $\mathcal C_p$ does not see the corresponding point of $\mathcal D_{\mathcal C}$.

We note the following important consequence, which is known since Kodaira-Spencer works on deformations in the sixties, but which is still frequently overlooked.

\begin{proposition}
\label{TeichHopfnotAna}
The Teichm\"uller space $\mathcal T(\mathbb S^3\times \mathbb S^1)$ cannot be endowed with the structure neither of an analytic space nor of a non-Hausdorff analytic space.
\end{proposition}

\begin{proof}
An analytic space, even non-Hausdorff, is locally Hausdorff since it is locally modelled onto the zero set of some holomorphic functions in $\C^n$. This contradicts \eqref{topoTeichHopf} and the subsequent discussion.\qed
\end{proof}
 
\end{example}
\section{Artin analytic stacks}
\label{Artin}
As shown by Proposition \ref{TeichHopfnotAna}, the Teichm\"uller space does not always admit the structure of an analytic space, even locally. Hence, to see it as an analytic object, one needs to use the more general notion of analytic stacks. 

Let $\mathcal A$ be the category of (complex) analytic spaces and morphisms. We consider the euclidean topology on the analytic spaces. Especially, a covering is just an open covering for the euclidean topology. Fix some smooth manifold $M$ as in Section \ref{Intro}. By a $M$-deformation, we understand a smooth morphism $\mathcal X\to A$ over an analytic space all of whose fibers are complex manifolds diffeomorphic to $M$.

We consider the contravariant functor $\underline{\mathcal M}(M)$ from $\mathcal A$ to the category of groupoids\footnote{Recall that a groupoid is a category all of whose morphisms are invertible.} such that 
\begin{enumerate}
\item For $A$ an analytic space, $\underline{\mathcal M}(M)(A)$ is the groupoid formed by $M$-deformations over $A$ (objects) and isomorphisms of $M$-deformations (morphisms).
\item For $f$ a morphism from $A$ to $B$, $\underline{\mathcal M}(M)(f)$ is the natural pull back mapping from $\underline{\mathcal M}(M)(B)$ to $\underline{\mathcal M}(M)(A)$.
\end{enumerate}
This functor satisfies several properties, in particular 
\begin{enumerate}
\item {\bf Descent.} Given an open covering $(A_\alpha)$ of $A$, and $M$-deformations $\mathcal X_\alpha$ over $A_\alpha$ and a cocycle of isomorphisms over the intersections, there exists a unique $\mathcal X$ over $A$ obtained by gluing all the $\mathcal X_\alpha$.
\item {\bf Sheaf.} Given an open covering $(A_\alpha)$ of $A$, $M$-deformations $\mathcal X$ and $\mathcal X'$ over $A$, a collection of isomorphisms $f_\alpha$ over $A_\alpha$ between the $M$-deformations that coincide over the intersections glue into a unique morphism $f$.
\end{enumerate}
This is an example of {\it a stack over the site}\footnote{In fact, a stack is a $2$-functor but we will not go into that.} $\mathcal A$, cf. \cite{Fantechi}. This is the stack version of the moduli space \eqref{M3}. To obtain the stack version of \eqref{T}, we have to modify slightly the construction. A $M$-deformation over $A$ is smoothly a $M$ bundle over $A$ with structural group $\text{Diff}^+(M)$. The functor $\underline{\mathcal T}(M)$ associates to each $A$ the set of isomorphism classes of $M$-deformations over $A$ whose smooth bundle structure has structural group $\text{Diff}^0(M)$. Let us call {\it reduced} such a $M$-deformation.

Let us make a break. We propose to replace the Teichm\"uller space $\mathcal T(M)$ with a complicated contravariant functor $\underline{\mathcal T}(M)$ which describes all the $M$-deformations over analytic spaces with structural group $\text{Diff}^0(M)$. At first sight, there is no reason to do this. This is not even clear that we are dealing with something similar to the Teichm\"uller space.

The point here is that in many cases a stack can also be described by a single groupoid, called {\it atlas} or {\it presentation} of the stack. An atlas is far from being unique and we can look for a "nice" one. In our case we can choose as an atlas a groupoid that looks much more like the Teichm\"uller space and which is in fact an enriched version of it. To see this less theoretically, let us revisit the example of complex tori of dimension one.

\begin{example}
 \label{torusrevisited}
 We let again $M$ to be $\mathbb S^1\times\mathbb S^1$ with a fixed orientation. Then $\underline{\mathcal T}(M)$ describes all reduced deformations with complex tori as fibers and their morphisms. We claim that an atlas for this stack can be obtained as follows. Consider the universal family $\mathcal X\to\mathbb H$ where $\mathcal X$ is the quotient of $\C\times\mathbb H$ by the free and proper holomorphic action
 \begin{equation}
  \label{uf}
  (z,\tau,p,q)\in\C\times\mathbb H\times\Z\times\Z\longmapsto (z+p+q\tau,\tau)
 \end{equation}
This is a reduced $\mathbb S^1\times\mathbb S^1$-deformation. The fiber over $\tau$ is the complex torus $\mathbb E_\tau$ of lattice $(1,\tau)$. We rewrite this family as the following groupoid
\begin{equation}
 \label{torusgroupoid}
 \mathcal X\rightrightarrows \mathbb H
\end{equation}
with both arrows equal to $\pi$. This must be thought of as follows. The set of objects is $\mathbb H$, that is the Teichm\"uller space $\mathcal T(\mathbb S^1\times\mathbb S^1)$. The set of morphisms is $\mathcal X$ and the two maps are the projection on the source and the target of the morphism. In \eqref{torusgroupoid}, since both equals $\pi$, every morphism has same source and target $\tau$. The set of morphisms above $\tau$ is $\mathbb E_\tau$ and represents the group of translations of the complex torus $\mathbb E_\tau$. Here, we set the projection of $(0,\tau)\in\C\times\mathbb H$ in $\mathcal X$ to be the zero translation. Hence our groupoid is just $\mathcal T(\mathbb S^1\times\mathbb S^1)$ with its translation group above each complex torus $\tau$.

We will not prove our claim, this would force us to give many and many definitions, but we can give a heuristic interpretation of it. The key point is that if you are given an analytic space $A$, then every torus deformation above $A$ can be recovered from \eqref{torusgroupoid}. From the one hand, above a sufficiently small open set $A_\alpha$ of $A$, such a deformation is completely and uniquely up to isomorphism characterized by a holomorphic map from $A_\alpha$ to $\mathbb H$\footnote{This is the meaning of universal family in Kodaira-Spencer deformation theory.}. From the other hand, gluings of these families over $A_\alpha$ are completely characterized by maps from $A_\alpha\cap A_\beta$ to $\mathcal X$ since such gluings are translations along the fibers of the deformation.

Note that \eqref{torusgroupoid} is a complex Lie groupoid: its sets of objects and morphisms are complex manifolds and the source and target maps are holomorphic submersions. 
\end{example}

In the general case, an atlas for $\underline{\mathcal T}(M)$ is given by the action groupoid
\begin{equation}
\label{atlasgeneral}
\text{Diff}^0(M)\times \mathcal I(M)\rightrightarrows \mathcal I(M)
\end{equation}
with source map $s$ and target map $t$ defined as
\begin{equation}
\label{structuremaps}
s(f, J)=J\qquad\text{ and }\qquad t(f,J)=J\cdot f
\end{equation}
with $J\cdot f$ defined as the right hand side of \eqref{Jf}. Since there is a dictionnary between a stack and an atlas for it, \eqref{atlasgeneral} explains why we say that $\underline{\mathcal T}(M)$ is the stack version of $\mathcal T(M)$. 

Compared with \eqref{torusgroupoid}, this is no longer a complex Lie groupoid but an infinite dimensional object. And it does not help us to understand the Teichm\"uller space since it is just a rewriting of \eqref{T}. But recall that an atlas is not unique. So introducing  $\underline{\mathcal T}(M)$ is interesting only if we can find a nice and useful atlas of it. By nice, we think of a complex Lie groupoid, but we will see it is too much too expect. We need a singular version of a complex Lie groupoid. So we define

\begin{definition}
\label{artindef}
A stack over the site $\mathcal A$ is an {\it Artin analytic stack} if it admits an atlas $A_1\rightrightarrows A_0$ with $A_0$ and $A_1$ being finite dimensional complex analytic spaces and the source and target maps being smooth morphisms. Such a groupoid is called a {\it singular Lie groupoid}.
\end{definition}

The main result of \cite{LMStacks} is to prove that, under a mild uniform condition on the automorphism group of $M$ when considered as a complex manifold, $\underline{\mathcal T}(M)$ is an Artin analytic stack and to construct an explicit atlas with the properties of Definition \ref{artindef}. This is the best possible answer to the question of Section \ref{ex}. The Teichm\"uller space $\mathcal T(M)$ is not an analytic space but its stack version  $\underline{\mathcal T}(M)$ is Artin analytic.

\section{Foliations and holonomy groupoid}
\label{foliation}
To understand the stack structure of $\underline{\mathcal T}(M)$, we need to make a diversion through foliation theory. Stacks and groupoids are also useful to understand the structure of the leaf space of a foliation.

So assume we start with a smooth foliation $\mathcal F$ of a smooth manifold $M$. Then $\mathcal F$ is defined through charts with values in $\R^p\times\R^{n-p}$ such that the changes of charts are of the type
\begin{equation}
(x,t)\in \R^p\times\R^{n-p}\longmapsto (g(x,t), h(t))\in \R^p\times\R^{n-p}
\end{equation}
that is send plaques $\{t=Cst\}$ onto plaques. Gluing the plaques following the changes of charts gives the leaves, that is disjoint immersed submanifolds that form a partition of  $M$. Transverse local sections to $\mathcal F$ are given by $\{x=Cst\}$.

\begin{example}
\label{torusfoliation}
An easy but yet interesting example is that of a linear foliation of a torus. Consider the trivial foliation of $\R^2$ given by parallel straight lines making a fixed angle $\alpha$ with the horizontal. It descends on the torus $\R^2/\Z ^2$ as a foliation by curves with leaves wrapping around the torus. Its properties depend on the arithmetic type of $\alpha$.
\begin{enumerate}
\item If $\alpha=p/q$ is rational then the leaves of the foliation are closed curves diffeomorphic to $\mathbb S^1$ that make $p$ turns in the vertical direction and $q$ in the horizontal one. The foliation has only compact leaves.
\item If $\alpha$ is irrational then the leaves are diffeomorphic to $\mathbb R$ and are dense in the torus. The foliation is minimal.
\end{enumerate}
If we look for the leaf space, we can fix a meridian $T$ on our torus. It is a global transverse to the foliation. Each leaf $L$ will cut the circle $T$ in at least one point. More precisely, the intersection $L\cap T$ is an orbit of the rotation of angle $\alpha$ and the leaf space is given by the quotient of $T$ by the group $G$ generated by this rotation. 
\begin{enumerate}
\item If $\alpha=p/q$ is rational then $G$ is isomorphic to $\Z_q$ and $T/G$ identifies with $\mathbb S^1$.
\item If $\alpha$ is irrational then $G$ is isomorphic to $\Z$ and $T/G$ is not Hausdorff.
\end{enumerate}
\end{example}

Let us see now how stacks and groupoids can help us in defining the leaf space.

In Example \ref{torusfoliation}, identifying $T$ with $\mathbb S^1\subset \mathbb C$ and $G$ with the group generated by the rotation $z\mapsto r_\alpha(z):=\exp{2i\pi\alpha}z$, we may encode this as the Lie groupoid 
$\langle r_{\alpha}\rangle\times \mathbb S^1\rightrightarrows \mathbb S^1$
with source and target maps given by
\begin{equation}
s(g,z)=z\qquad\text{ and }\qquad t(g,z)=g(z)
\end{equation}
This is an example of an action Lie groupoid: the source map is the projection onto the second factor and the target map is given by the action, cf. \eqref{structuremaps}. Of course, just doing this is not enough. We have to think of this groupoid as defining a stack over the category of smooth manifolds. And as such, it is not the only groupoid defining this stack. Hence we have to think as the leaf space not as $G\times T\rightrightarrows T$ but as an equivalence class of groupoids.

There exists a general construction to encode the leaf space in a Lie groupoid: the \'etale holonomy groupoid. Roughly speaking, one starts with a foliated atlas of $\mathcal F$ and define as objects of the groupoid the disjoint union of a complete sets of local transverse sections. Then the set of morphisms encodes the holonomy morphisms. This is quite technical to do and we refer to \cite[\S 5.2]{MM} for more details. 

As in Example \ref{torusfoliation}, this groupoid is an atlas for a stack over the category of smooth manifolds. And this stack has lots of different atlases, all equivalent.
We will not give the precise definition of the equivalence relation needed here\footnote{It is called Morita equivalence and basically is an adaptation of equivalence of category in the world of smooth manifolds.}. Let us just say that this equivalence class is an enriched version of the topological quotient. It does not remember neither $M$ nor $\mathcal F$ but encodes the topological leaf space and moreover the smoothness of the initial construction (since it remains in the realm of Lie groupoids) and all the holonomy data. This is the best definition of a leaf space. In the case of Example \ref{torusfoliation}, if $\alpha$ is rational, then $G\times T\rightrightarrows T$ is equivalent to the trivial groupoid $\mathbb S^1\rightrightarrows \mathbb S^1$, that is the stack is just the manifold $\mathbb S^1$. However, if $\alpha$ and $\alpha'$ are irrational, then the corresponding Lie groupoids $G\times T\rightrightarrows T$ are equivalent if and only if $\alpha'=A\cdot \alpha$ as in \eqref{SL2action}, see \cite{Rieffel}. Stacks allow to distinguish the leaf spaces in the irrational case.

Of course everything works in the analytic context. If the foliation is holomorphic, then the \'etale holonomy groupoid is a complex Lie groupoid and defines a stack over the category of complex manifolds. If we look at regular foliations (i.e. leaves are manifolds) on a singular space, then everything works except that the \'etale holonomy groupoid is now a singular Lie groupoid  and the stack is Artin analytick in the sense of Definition \ref{artindef}.

\section{The Teichm\"uller stack}
\label{TeichStack}

We are now in position to give the main results of \cite{LMStacks} and the main ingredients of the proof. As before, we denote by $M$ a compact connected oriented even-dimensional smooth manifold, by $\mathcal T(M)$ its Teichm\"uller space.

\begin{definition}
\label{Tstack}
We call {\it Teichm\"uller stack} the stack $\underline{\mathcal T}(M)$ defined in Section \ref{Artin}.
\end{definition}

We have

\begin{theorem}[cf. \cite{LMStacks}]
\label{maintheorem}
Assume that there exists a constant $a\in\mathbb N$ such that, for all $J\in \mathcal T(M)$, the dimension of the automorphism group of the corresponding complex manifold $X_J$ is bounded by $a$.

Then $\underline{\mathcal T}(M)$ is an Artin analytic stack.
\end{theorem}

Before explaining the main ideas of the proof, some important remarks have to be done.
\begin{enumerate}
\item The proof is constructive and geometric. It builds a concrete singular Lie groupoid as atlas for $\underline{\mathcal T}(M)$ which comes from the existence of a geometric structure (a foliated structure) of $\mathcal I(M)$. This is perhaps the most interesting aspect of the result.
\item The hypothesis is used to control that the constructed atlas is finite-dimensional. In any case this is a mild restriction since we may easily stratify $\mathcal T(M)$ by strata satisfying the hypotheses for a given $a$. Classical results of Grauert ensure that this gives a nice analytic stratification, see \cite{LMStacks} for more details.
\end{enumerate}

The crucial idea is to understand that the action of $\text{Diff}^0(M)$ defines  a "generalized foliation" on $\mathcal I(M)$ so that the construction of a "generalized \'etale holonomy groupoid" can be carried out. To do this, many technical problems have to be overcome, the most serious one being the presence of non-trivial automorphisms (remark that the isotropy groups of the actions are constituted by automorphisms).

\medskip

\noindent $\underline{\text{1st case:}} \text{ no automorphisms.}$

\smallskip

Here we assume that, for all $J\in \mathcal T(M)$, we have 
\begin{equation}
\label{noaut}
\text{Aut }(X_J)\cap\text{Diff}^0(M)=\{Id\}
\end{equation}

Hypothesis \eqref{noaut} exactly means that the $\text{Diff}^0(M)$ action is free. It is thus natural to expect that it defines a foliation in some sense. Now, this is a reformulation of Kuranishi's theorem of existence of a versal space \cite{Kur}.

\begin{theorem}[Kuranishi, 1962]
\label{KuranishiNoAut}
Let $X_0=(M,J_0)$ be a compact complex manifold whose underlying smooth structure is $M$. Assume {\rm \eqref{noaut}}. Then there exists a finite-dimensional analytic space $K_0$ such that the space $\mathcal I(M)$ is locally isomorphic to $K_0\times \text{\rm Diff}^0(M)$ in a neighborhood of $J_0$.
\end{theorem}

Of course this local isomorphism preserves locally the action, so this gives really a foliated chart for the action. The plaques are open neighborhoods of the identity in the Fr\'echet manifold $\text{Diff}^0(M)$ and a transverse local section is given by the analytic space $K_0$. So this is an infinite dimensional but of finite codimension foliation of an infinite-dimensional analytic space, cf. Section \ref{Intro}, just before Remark \ref{cc}.

In this situation, we can carry out the construction of the \'etale holonomy groupoid. Only slight adaptations have to be done. Note that the finite codimension of the foliation ensures the finite dimensionality of the holonomy groupoid. 

\medskip

\noindent $\underline{\text{2nd case:}}\text{ general case.}$
Kuranishi's theorem takes now the following more general form.
\smallskip

\begin{theorem}[Kuranishi, 1962]
\label{Kuranishi}
Let $X_0=(M,J_0)$ be a compact complex manifold whose underlying smooth structure is $M$. Let $\text{\rm Aut}^0(X_0)$ be the connected component of the identity in the automorphism group of $X_0$. Then 
\begin{enumerate}
\item[\rm  1.] A neighborhood of the identity in the quotient space $(\text{\rm Diff}^0(M)/\text{\rm Aut}^0(X_0))$ is a Fr\'echet manifold.
\item[\rm 2.] There exists a finite-dimensional analytic space $K_0$ such that the space $\mathcal I(M)$ is locally isomorphic to $K_0\times (\text{\rm Diff}^0(M)/\text{\rm Aut}^0(X_0))$ in a neighborhood of $J_0$.
\end{enumerate}
\end{theorem}

As in the previous case, this local isomorphism preserves locally the action, but this time this  does not give a foliated chart for the action. The problem is that the plaques are now modelled onto $(\text{\rm Diff}^0(M)/\text{\rm Aut}^0(X_0))$, i.e. depends on the automorphism group of the base manifold $X_0$. Hence plaques of different charts cannot be glued. There is no leaf to be constructed from the plaques.

Now we can reformulate Theorem \ref{Kuranishi} as giving a local isomorphism at $J_0$ between $\mathcal I(M)$ and the product
\begin{equation}
\label{reloaded}
\text{\rm Diff}^0(M)\times [K_0/\text{\rm Aut}^0(X_0)]
\end{equation}
where the brackets mean that we consider the right hand side as an Artin analytic stack. In other words, $\text{\rm Aut}^0(X_0)$ acts\footnote{In fact, this is not exactly an action, see \cite{LMStacks} for more details.} on $K_0$ and we consider its quotient as a stack. In foliated terms, we force the plaques to be open neighborhoods of the identity in the Fr\'echet manifold $\text{Diff}^0(M)$. This is possible subject to the condition that we let the transverse sections to be analytic stacks rather than analytic spaces.

Then \eqref{reloaded} can be interpreted as a foliated chart in a generalized sense and the gluings will respect this foliated structure. In \cite{LMStacks}, we call it a {\it foliation transversely modelled on a translation groupoid} or in short a {\it TG foliation}. 

The next step consists in showing how to adapt the machinery of holonomy \'etale groupoid to the world of TG foliations. This forms the technical core of \cite{LMStacks}. The gap with the classical theory is important and lots of work is needed. We will not get into that, since we attain our assigned goal: give a comprehensive introduction to the results and objects of \cite{LMStacks}.

\begin{remark}
\label{Artinvsetale}
In the realm of Section \ref{foliation}, we get an \'etale holonomy groupoid. This is more than a singular Lie groupoid since the source and target maps are not only smooth morphisms but also \'etale morphisms. In the general case, the holonomy groupoid associated to a TG foliation has no more this property. 
\end{remark}

\section{The Teichm\"uller stack of Hopf surfaces}
\label{appli}
In this last section, we shall briefly report on work in progress by C. Fromenteau on the Teichm\"uller stack of $\mathbb S^3\times\mathbb S^1$. We explain in Example \ref{Hopfexample} the quite complicated topology of the Teichm\"uller {\it space} of $\mathbb S^3\times\mathbb S^1$ as well as the different normal forms for the associated complex structures. It is known that the automorphism group of a Hopf surface has dimension $2$, $3$ or $4$ depending on the normal form. Hence the hypothesis of Theorem \ref{maintheorem} is fulfilled and $\underline{\mathcal T}(\mathbb S^3\times\mathbb S^1)$ is an Artin analytic stack. The construction of the holonomy groupoid refered to in Section \ref{TeichStack} does not give a useful atlas in this case. In particular its set of objects has countably many connected components. To really work with $\underline{\mathcal T}(\mathbb S^3\times\mathbb S^1)$ we need another atlas.

C. Fromenteau gave a much nicer atlas that can be described as follows. Let $G$ be the Lie group biholomorphic to $\text{GL}_2(\C)\times\C$ as a complex manifold but with the following product rule
\begin{equation}
\label{product}
(A,t)\ast (B,s)=(AB, t+s\det A)
\end{equation}
Let $M$ be the product $\text{GL}_2^c(\C)\times\C$. Then one may define 
\begin{enumerate}
\item a holomorphic action $\cdot$ of $G$ onto $M$.
\item a holomorphic injection $\imath$ of $M$ into $G$
\end{enumerate}
such that  the Lie groupoid $(G\times M)/\Z\rightrightarrows M$ is an atlas of $\underline{\mathcal T}(\mathbb S^3\times\mathbb S^1)$. Here the $\Z$-action is defined as
\begin{equation}
\label{Zaction}
(p,g,m)\in\Z\times G\times M\longmapsto (\imath (m)^pg, m)
\end{equation} 
and the source and target maps are the projections of the maps\footnote{The action groupoid $G\times M \rightrightarrows M$ with source and target maps defined in \eqref{onceagain} is an atlas for the stack of reduced $\mathbb S^3\times \mathbb S^1$-deformations admitting a covering $\C^2\setminus \{(0,0)\}$-deformation plus a choice of a base point in the covering family. Together with $\imath$, this forms a gerbe with band $\Z$.} 
\begin{equation}
\label{onceagain}
(g,m)\in G\times M\longmapsto m\qquad\text{ and }\qquad (g,m)\mapsto m\cdot g
\end{equation}
 One interest of this atlas is for cohomological computations. For example,  $\underline{\mathcal T}(\mathbb S^3\times\mathbb S^1)$ has well defined de Rham cohomology groups, cf. \cite{Behrend}. This cohomology is different from the cohomology of the topological space ${\mathcal T}(\mathbb S^3\times\mathbb S^1)$ since it takes also into account the cohomology of the automorphism groups of Hopf surfaces. 
 
 In any case, these cohomology groups are very difficult to compute in general. Having such an atlas makes the calculations possible. Roughly speaking, they are just the equivariant cohomology groups of the action $\cdot$ of $G$ onto $M$. Here, with more work, one can compute them and show that the generators in dimension $2$ can be realized as non-trivial holomorphic bundles above $\mathbb P^1$ with fiber a Hopf surface (which must be thought of as isotrivial but not trivial $\mathbb S^1\times\mathbb S^3$-deformation above $\mathbb P^1$). It is not clear however whether the generator of the first cohomology group (which is equal to $\C$) can be realized as a {\it holomorphic} $\mathbb S^1\times\mathbb S^3$-deformation above a compact Riemann surface.

\begin{acknowledgement}
Many thanks to Daniele Angella, Paolo de Bartolomeis, Costantino Medori and Adriano Tomassini for organizing this beautiful conference in Cortona.
\end{acknowledgement}

\end{document}